\documentclass[10pt]{article}

\title{Random embeddings of bounded degree trees with optimal spread}
\author{Paul Bastide \and Clément Legrand-Duchesne \and Alp Müyesser}

\usepackage{graphicx, caption}
\usepackage{amssymb}
\usepackage{amsthm}
\usepackage{amsmath}
\usepackage{enumitem}
\usepackage{framed} 
\usepackage{dsfont}

\usepackage{soul}
\usepackage{changepage}
\usepackage{comment}
\usepackage{float}
\usepackage[dvipsnames, table]{xcolor}
\definecolor{darkblue}{rgb}{0,0,0.5}
\usepackage[
    colorlinks=true,
    linkcolor=darkblue,
    citecolor=darkblue]{hyperref}
\usepackage[capitalise]{cleveref}
\usepackage{marginnote}
\usepackage{lscape}

\usepackage{tikz}
\usetikzlibrary{calc}
\usetikzlibrary{shapes}
\usetikzlibrary{snakes}
\usetikzlibrary{decorations.pathreplacing,calligraphy}

\usepackage{mathtools}

\usepackage[ruled,vlined]{algorithm2e}

\SetCommentSty{mycommfont}

\theoremstyle{plain}
\newtheorem{theorem}{Theorem}[section]
\newtheorem{lemma}[theorem]{Lemma}

\newtheorem{proposition}[theorem]{Proposition}
\newtheorem{claim}{Claim}[theorem]

\newtheorem{corollary}[theorem]{Corollary}
\theoremstyle{definition}
\newtheorem{definition}[theorem]{Definition}

\newtheorem{remark}{Remark}

\usepackage{mdframed}
\usepackage{lipsum}

\newmdtheoremenv{definbox}[theorem]{Definition}

\newcommand\eps{\varepsilon}
\newcommand\N{\mathbb{N}}
\newcommand\E{\mathbb{E}}
\renewcommand\P{\mathbb{P}}

\newcommand{\cG}{\ensuremath{\mathcal G}}

\newcommand{\Prob}[1]{\ensuremath{%
    \mathbb P\left[#1\right]
  }}
  
\newcommand{\Expect}[1]{\ensuremath{%
    \mathbb E\left[#1\right]
  }}

\newcommand{\cF}{\mathcal{F}}

\renewcommand{\ge}{\geqslant}
\renewcommand{\le}{\leqslant}
\renewcommand{\geq}{\geqslant}
\renewcommand{\leq}{\leqslant}
\DeclareMathOperator{\Ima}{Im}

\newcommand{\doublesquig}{%
  \mathrel{%
    \vcenter{\offinterlineskip
      \ialign{##\cr$\rightsquigarrow$\cr\noalign{\kern-1.5pt}$\rightsquigarrow$\cr}%
    }%
  }%
}

\DeclareMathOperator{\gnp}{\cG}
\DeclareMathOperator{\gknp}{\cG^{(\emph{k})}}

\newcommand{\aside}[1]{\marginnote{#1}}

\usepackage{graphicx}%

\newcommand{\PB}[1]{\textcolor{ForestGreen}{ $\blacktriangleright$\ {\sf PB: #1}\
  $\blacktriangleleft$ }}

\AtBeginDocument{%
   \def\MR#1{}
} 

\begin{document}
\maketitle
\begin{abstract}
    A seminal result of Koml\'os, S\'ark\"ozy, and Szemer\'edi states that any $n$-vertex graph $G$ with minimum degree at least $(1/2+\alpha)n$ contains every $n$-vertex tree $T$ of bounded degree. Recently, Pham, Sah, Sawhney, and Simkin extended this result to show that such graphs $G$ in fact support an \textit{optimally spread distribution} on copies of a given $T$, which implies, using the recent breakthroughs on the Kahn-Kalai conjecture, the \textit{robustness} result that $T$ is a subgraph of sparse random subgraphs of $G$ as well.
    \par Pham, Sah, Sawhney, and Simkin construct their optimally spread distribution by following closely the original proof of the Koml\'os-S\'ark\"ozy-Szemer\'edi theorem which uses the \text{blow-up lemma} and the Szemer\'edi regularity lemma. We give an alternative, regularity-free construction that instead uses the Koml\'os-S\'ark\"ozy-Szemer\'edi theorem (which has a regularity-free proof due to Kathapurkar and Montgomery) as a black-box.
    \par Our proof is based on the simple and general insight that, if $G$ has linear minimum degree, almost all constant sized subgraphs of $G$ inherit the same minimum degree condition that $G$ has.   
\end{abstract}

\section{Introduction}

There is a large body of results in extremal graph theory focusing on determining the minimum degree threshold which forces the containment of a target subgraph. For example, a classical result of Dirac \cite{dirac1952} states that any $n$-vertex graph with minimum degree at least $n/2$ contains a Hamilton cycle. Although this result is tight, graphs with minimum degree $n/2$ are quite dense, so it is natural to suspect that they are Hamiltonian in a rich sense. In this direction, S\'arközy, Selkow, and Szemer\'edi \cite{SSS2003} showed that $n$-vertex graphs with minimum degree $n/2$ contain $\Omega(n)^n$ distinct Hamilton cycles (we refer to such results as \textit{enumeration} results, see also \cite{CK2009}). Moreover, randomly sparsifying the edge set of an $n$-vertex graph with minimum degree $n/2$ yields, with high probability, another Hamiltonian graph, as long as each edge is kept with probability $\Omega(\log n /n)$. This follows from an influential result of Krivelevich, Lee, and Sudakov \cite{krivelevich2014robust} (such results are referred to as \textit{robustness} results, see \cite{sudakov2017robustness}), which generalises P\'osa's celebrated result stating that the random graph $G(n,C\log n/n)$ is Hamiltonian with high probability. 
\par The study of random graphs was recently revolutionised by Frankston, Kahn, Narayanan, and Park's \cite{FKNP21} proof of the \textit{fractional expectation threshold vs.\ threshold} conjecture of Talagrand~\cite{Ta10} (see also \cite{park2022proof}  for a proof of the even stronger Kahn--Kalai conjecture~\cite{KK07}). In our context, these breakthroughs imply that the enumeration and robustness results stated in the first paragraph, which themselves are fairly general, admit a a further common generalisation. To state this generalisation, we need the language of \textit{spread distributions} which we will define momentarily. In a nutshell, the key idea is to show that a graph $G$, with minimum degree large enough to necessarily contain a copy of a target graph $H$, actually supports a \textit{random embedding} of $H$ that (roughly speaking) resembles a uniformly random function from $V(H)$ to $V(G)$. The formal definition we use is below.
\begin{definition}[\cite{PSSS22}]\label{def:maindef}
    Let $X$ and $Y$ be finite sets, and let $\mu$ be a probability distribution over injections $\varphi : X \rightarrow Y$.  
    For $q \in [0, 1]$, we say that $\mu$ is \textit{$q$-spread} if for every two sequences of distinct vertices $x_1, \dots, x_s \in X$ and $y_1, \dots, y_s \in Y$,
    \begin{equation*}           
    \mu\left(\left\{\varphi : \varphi(x_i)=y_i \text{ for all }i \in [s]\right\}\right) \leq q^s.
    \end{equation*}    
\end{definition} In our context, $X=V(H)$, $Y=V(G)$, $|X|=|Y|=n$, and $\mu$ is a probability distribution over embeddings of $H$ into $G$. The gold standard for us is constructing distributions $\mu$ that are $O(1/n)$-spread. Such distributions have the optimal spread (up to the value of the implied constant factor) that is also attained by a random injection from $V(H)$ to $V(G)$. We remark that Definition~\ref{def:maindef}, originally introduced in \cite{PSSS22}, is different to the usual definition of spreadness phrased in terms of edges, instead of vertices. However, for embedding spanning subgraphs, the above definition turns out to be more convenient (see \cite{kelly2023optimal, PSSS22} for more details).
\par The breakthroughs on the Kahn-Kalai conjecture have created a lot of incentive to show ``spread versions'' of Dirac-type results in graphs and hypergraphs, as such results directly imply enumeration and robustness results, thereby coalescing two streams of research which have, until now, been investigated independently. We refer the reader to the recent papers \cite{PSSS22, kelly2023optimal, kang2024perfect, joos2023robust,anastos2023robust} that obtain several results in this direction (see also \cite{allen2024robust}). Most of the aforementioned work focuses on constructing spread distributions for target graphs with rather simple structures, such as perfect matchings or Hamilton cycles. One notable exception is the result from \cite{PSSS22} for bounded degree trees. To introduce this result, we first cite the following classical result in extremal graph theory.
\begin{theorem}[Koml\'os--S\'ark\"ozy--Szemer\'edi~\cite{KSS95}]\label{thm:KSS95}\ For every $\Delta\in\mathbb N$ and $\alpha>0$,
there exists $n_0\in\mathbb N$ such that the following holds for all $n\ge n_0$.
If $G$ is an $n$-vertex graph with $\delta(G)\ge (1+\alpha)\frac{n}{2}$, 
then $G$ contains a copy of every $n$-vertex tree with maximum degree bounded by $\Delta$.
\end{theorem}
Theorem~\ref{thm:KSS95} admits a spread version, as demonstrated in \cite{PSSS22}. 
\begin{theorem}[Pham, Sah, Sahwney, Simkin \cite{PSSS22}]\label{thm:mainthm} For every $\Delta\in\mathbb N$ and $\alpha>0$,
there exists $n_0, C\in\mathbb N$ such that the following holds for all $n\ge n_0$. If $G$ is an $n$-vertex graph with $\delta(G)\ge (1+\alpha)\frac{n}{2}$, and $T$ is a $n$-vertex tree with $\Delta(T)\leq \Delta$, there exists a $(C/n)$-spread distribution on embeddings of $T$ onto $G$.
\end{theorem}
Using the $s=n$ case of Definition~\ref{def:maindef}, Theorem~\ref{thm:mainthm} allows us to deduce that in the context of Theorem~\ref{thm:KSS95}, $G$ contains $\Omega(n)^n$ copies of a given bounded degree tree (see \cite{joos2024counting} for a more precise result). Furthermore, Theorem~\ref{thm:mainthm} implies that the random subgraph $G'\subseteq G$ obtained by keeping each edge of $G$ with probability $\Omega(\log n/n)$ also contains a given bounded degree tree (see \cite{PSSS22} for a precise statement).
\par The original proof \cite{KSS95} of Theorem~\ref{thm:KSS95} constitutes one of the early
applications of the Szemer\'edi regularity lemma (used in conjunction with the blow-up
lemma of Koml\'os, S\'ark\"ozy, and Szemer\'edi). The proof of the more general
Theorem~\ref{thm:mainthm} in \cite{PSSS22} can be interpreted as a randomised
version of the proof in \cite{KSS95}. Indeed, readers familiar with applications
of the regularity/blow-up lemma would know that whilst embedding a target
subgraph with this method, there is actually a lot of flexibility for where each vertex can go. Thus, a choice can be made randomly from the available options as a reasonable strategy towards proving Theorem~\ref{thm:mainthm}.
\par The main contribution of the current paper is a proof of Theorem~\ref{thm:mainthm} that uses Theorem~\ref{thm:KSS95} as a black-box. The most obvious advantage of such a proof is that, as Theorem~\ref{thm:KSS95} has a more modern proof due to Montgomery and Kathapurkar \cite{kathapurkar2022spanning} that circumvents the use of Szemer\'edi regularity lemma, our proof yields a regularity-free proof of Theorem~\ref{thm:mainthm} which naturally has better dependencies between the constants (see Remark~\ref{rem:mainrem}). Our proof is presented in Section~\ref{sec:mainsec}, and Section~\ref{sec:overview} contains an overview explaining the key ideas.
\subsection{Future directions}\label{sec:futuredirections}
\par Our proof can also be interpreted as modest progress towards a more ambitious research agenda, hinted to in \cite{kelly2023optimal}, which asserts that \textit{all} Dirac-type results admit a spread version, regardless of the target structure being embedded. One reason why such a general result could hold is that in Dirac-type results, host graphs have linear minimum degree. Thus, Chernoff's bound can be used to show that almost all $O(1)$-sized induced subgraphs of such dense host graphs maintain the same (relative) minimum degree condition. If the target graph itself has some recursive structure, we may use this to our advantage whilst constructing a random embedding with optimal spread. The strategy would be to first break up the target graph into pieces of size $O(1)$, for example, in the case of Hamilton cycles, we would simply break up the cycle into several subpaths. For each $O(1)$-sized subpath, almost all $O(1)$-sized subsets of the host graph have large enough minimum degree to necessarily contain a copy of the subpath (simply by invoking Dirac's theorem), so we may choose one such host subset randomly while constructing a random embedding with good spread. 
\par A variant of the above strategy was successfully implemented in \cite{kelly2023optimal} in the context of hypergraph Hamilton cycles. In this paper, we show that a similar strategy works even for bounded-degree spanning trees, which, like Hamilton cycles, have a recursive structure (see Section~\ref{sec:splittings}), albeit a lot more complex than that of Hamilton cycles. We believe our methods are fairly general, and they could translate to construct spread distributions in the context of directed trees \cite{kathapurkar2022spanning}, hypertrees \cite{pehova2024embedding, pavez2024dirac}, or other related structures such as spanning grids. 
\par It remains an interesting open problem to find an even larger class of target graphs for which the Dirac-type theorem admits a spread generalisation. For example, it would be natural to investigate graph families with sublinear bandwidth, and we believe our methods could be applicable here. Note that this would entail more than simply randomising the blow-up lemma based proof of the Bandwidth Theorem \cite{bottcher2009proof}, as this theorem does not always give the optimal minimum degree condition for the containment of every graph family with sublinear bandwidth. Though, of course, obtaining a spread version of the bandwidth theorem would be of independent interest.

Koml\'os, S\'ark\"ozy, Szemer\'edi \cite{komlos2001} actually proved a stronger result than Theorem~\ref{thm:KSS95} where the maximum degree hypothesis is relaxed as $\Delta(T)=o(n/\log n)$. It would be interesting to similarly strengthen Theorem~\ref{thm:mainthm} by weakening the assumption on $\Delta(T)$. 
\section{Preliminaries}

We use the standard notation for hierarchies of constants, writing $x\ll y$ to mean that there exists a non-decreasing function $f : (0,1] \rightarrow (0, 1]$ such that the subsequent statements hold for $x\leq f(y)$. Hierarchies with multiple constants are defined similarly.

We will use the following theorem to embed our subtrees of bounded size. It generalises \cref{thm:KSS95} and was proved in \cite{alp_yani_richard}, using tools from \cite{kathapurkar2022spanning}. In particular, the proof of the following theorem does not rely on the Szemer\'edi regularity lemma.

\begin{theorem}[Theorem 4.4 \cite{alp_yani_richard}]\label{thm:rootedtree}
Let $1/n\ll 1/\Delta,\alpha$. Let $G$ be an $n$-vertex graph with $\delta(G)\geq (1/2+\alpha)n$. Let $T$ be an $n$-vertex tree with $\Delta(T)\leq \Delta$. Let $t\in V(T)$ and $v\in V(G)$. Then, $G$ contains a copy of $T$ with $t$ copied to $v$.
\end{theorem}

\subsection{Tree-splittings}\label{sec:splittings}
\begin{definition}
  Let $T$ be an $n$-vertex tree. A \emph{tree-splitting} of size $\ell$ is a
  family of edge-disjoint subtrees $(T_i)_{i\in [\ell]}$ of $T$ such that
  $\bigcup_{i \in [\ell]} E(T_i) = E(T)$. Note that for any $i \neq j$, the
  subtrees $T_i$ and $T_j$ intersect on at most one vertex. Given a tree-splitting $(T_i)_{i \in [\ell]}$ of an $n$-vertex tree $T$, we
  define the \emph{bag-graph} of the tree-splitting to be the graph whose nodes
  are indexed by $[\ell]$ and in which the nodes $i$ and $j$ are adjacent if
  $V(T_i) \cap V(T_j) \neq \emptyset$. A \emph{bag-tree} of a
  tree-splitting is simply a spanning tree of the bag-graph.
\end{definition}



We will use the following simple proposition to divide a tree into subtrees (see, for example, \cite[Proposition~3.22]{randomspanningtree}).

\begin{proposition}\label{littletree} Let $n,m\in \N$ such that $1\leq m\leq n/3$. Given any $n$-vertex tree~$T$ containing a vertex $t\in V(T)$, there are two edge-disjoint trees $T_1,T_2\subset T$ such that $E(T_1)\cup E(T_2)=E(T)$, $t\in V(T_1)$ and $m\leq |T_2|\leq 3m$.
\end{proposition}

This implies that a tree can be divided into many pieces of roughly equal size, as follows (see \cite{alp_yani_richard} for the simple proof).

\begin{corollary}\label{treedecomp}
Let $n,m\in \N$ satisfy $m\leq n$. Given any $n$-vertex tree~$T$, there exists a tree-splitting $(T_i)_{i \in [\ell]}$ of $T$ such that for each $i\in [\ell]$, we have $m\leq |T_i|\leq 4m$.
\end{corollary}




\subsection{Probabilistic results}

Below we give a lemma that encapsulates a simple argument often used when computing the spreadness of a random permutation.

\begin{lemma}
    \label{lem:perm_spread}
    Let $n \in \mathbb{N}$, $s \leq n$, and $L_1,\ldots,L_s \subseteq [n]$. For any distinct integers $1 \leq x_1,\ldots,x_s \leq n$, a uniformly sampled permutation $\pi$ of $[n]$ satisfies 
    $
    \P\left[\bigwedge_{i=1}^s \pi(x_i) \in L_i\right] \leq \prod_{i=1}^s \frac{e|L_{i}|}{n}.
    $
\end{lemma}
\begin{proof}[Proof of \cref{lem:perm_spread}]
    We have the following,
    \begin{align*}
         \P\left[\bigwedge_{i=1}^s \pi(x_i) \in L_i\right] &= \prod_{i=1}^s \P\left[ \pi(x_i) \in L_i \middle| \bigwedge_{j=1}^{i-1} \pi(x_j) \in L_j \right] 
          \leq \prod_{i=1}^s \frac{|L_i|}{n-i+1} \leq \prod_{i=1}^s \frac{e|L_{i}|}{n},
    \end{align*}
    where in the last step we used the fact that $\prod_{i=1}^s n-i+1 \geq \left(\frac{n}{e}\right)^s$, which is a well-known application of Stirling's approximation.
\end{proof}

Next, we present a lemma that records several properties we need from a random vertex-partition of a dense graph. The proof consists of standard applications of well-known concentration inequalities, namely Chernoff's bound and McDiarmid's inequality, and can be found in Appendix~\ref{app:b}.

\begin{lemma}\label{lem:randompartition2} Let $1/n\ll \eta, 1/C, 1/K, \delta,
  \alpha$, and suppose $1/C \ll 1/K \ll \alpha$. Fix two sequences of integers $(a_c)_{C\le c \le 4C}\geq 1$ and $(b_c)_{C\le c\le 4C}\geq 1$
  such that $\sum_{C \le c \le 4C} b_ca_c<n$, $b_c\geq \eta n$ and $C-K \leq a_c 
  \leq 4C-K$ for each $C \le c \le 4C$. Set $\eps := e^{-\alpha^2C/12}$. Let $G$ be a
  $n$-vertex graph with $\delta(G)\geq (\delta+\alpha)n$ and let $v \in V(G)$. Then, there exists a random labelled partition $\mathcal{R} = (R_c^j)_{C \le c \le 4C, j \in [b_c]}$ of a subset of $V(G)\setminus \{v\}$ into $\sum_{C\le c \le 4C} b_c$ parts, with the following properties,
    \begin{enumerate}[label ={{\emph{\textbf{A\arabic{enumi}}}}},
      ref=\textbf{A\arabic{enumi}}]
    \item \label{cond:A0} $\forall c \in [C, 4C], \forall j \in [b_c], |R_c^  j| = a_c$, hence there are exactly $b_c$ parts of size $a_c$;
    \item $\forall R_c^j \in \mathcal{R}, \, |\{ u \in V(G) \, | \,  \deg(u,R_c^j) \geq  (\delta + \frac{\alpha}{2}) |R_c^j+u|\}| \geq (1 - 3e^{-\frac{\alpha^2C}{10}}) |V(G)|;$   \label{cond:A1}
    \item  $\forall u \in V(G), \, |\{ R_c^j \in \mathcal{R} \, | \,  \delta(G[R_c^j+u]) \geq (\delta + \frac{\alpha}{2})|R_c^j|\}| \geq (1 - 3e^{-\frac{\alpha^2C}{10}}) |\mathcal{R}|.$ \label{cond:A2}
    \end{enumerate}

  Moreover, call $R_c^j\in \mathcal{R}$ good if $\delta(G[R_c^j])\geq (\delta+\alpha/2)|R_c^j|$. Call
  $R_c^j,R_d^k\in \mathcal{R}$ a good pair if
  $\delta(G[R_c^j+v])\geq (\delta+\alpha/2)|R_c^j+v|$ for each $v\in R_d^k$, and the same statement holds with $j$ and $c$ interchanged with $k$ and $d$. Let $A$ be the auxiliary graph with vertex set 
  $\mathcal{R}$ where $R_c^j\sim_A R_d^k$ if and only if $R_c^j,R_d^k$ is a good
  pair. Then, there exists a subgraph $ A'$ of $A$ such that
  \begin{enumerate}[label ={{\emph{\textbf{B\arabic{enumi}}}}},
      ref=\textbf{B\arabic{enumi}}]
  \item For all $c\in[C;4C], A'_c := \{ R_c^j \in A' : |R_c^j| = a_c \}$ has size at least $(1 - \eps) b_c$; \label{cond:B1}
  \item $ \forall R_c^j \in A', \forall d \in [C, 4C], \deg_{A'}(R_c^j, A'_d) \ge (1 -
    \eps)|A'_d|$; \label{cond:B2}
  \item $\forall R_c^j \in A'$, $R_c^j$ is good. \label{cond:B3}
  \end{enumerate}

  Furthermore, the following spreadness property holds. For any function
  $f\colon \{u_1,\ldots, u_s\}\to \mathcal{R}$ (where
  $\{u_1,\ldots, u_s\}\subseteq V(G)$),
  $\mathbb{P}[u_i\in f(u_i)]\leq (12C/n)^s$.

\end{lemma}
\subsection{Good spread with high minimum degree}
If the minimum degree of the host graph is larger than the size of the target graph we wish to embed, a simple random greedy algorithm can find a distribution with good spread. The following lemma records a version of this where vertices of the target and host graphs are coloured, and the embedding we produce respects the colour classes. This is used in the proof whilst (randomly) embedding the bag tree of a tree splitting into an auxiliary graph where vertices represent random subsets of a host graph, edges represent good pairs (as in Lemma~\ref{lem:randompartition2}), and the colours represent the size of the random set.
\begin{lemma}\label{lem:colourspread} Let $1/n\ll \gamma \ll \eta \leq 1$.  Let
  $G$ be a graph and $v \in V(G)$. Let $V_1\cup \ldots \cup V_t$ be a partition of
  $V(G) \setminus v$ where $|V_i| \geq \eta n$ for all $i\in [k]$ and so that
  for all $i\in [k]$, for all $u\in V(G)$, $\deg(u, V_j)\geq
  (1-\gamma)|V_j|$. Let $T$ be a tree and let $t$ be a vertex of $T$. Let $c$ be a
  $k$-colouring of $T-t$ such that the number
  of vertices of colour $i$ is at most $(1-\eta)|V_i|$. Then, there exists a random embedding $\phi\colon T\to G$ such that the following all hold.
  \begin{enumerate}
  \item With probability $1$, $\phi(t) = v$.
  \item With probability $1$, any $i$-coloured vertex of $T$ is embedded in $V_i$.
  \item The random embedding induced via $\phi$ by restricting to the
    forest $T-t$ is a $\left(\frac{2}{\eta^2 n}\right)$-spread embedding.
  \end{enumerate}
\end{lemma}
\begin{proof} 
Let us consider an ordering $t, t_1,t_2,\ldots,t_m$ of the vertices of $T$
rooted in $t$, where $T[\{t,t_1,\ldots, t_i\}]$ is a subtree for each $i\in[m]$. We define $\phi : T \rightarrow G$ greedily vertex by vertex
following the ordering of $V(T)$. Let $\phi(t):= v$. We denote by $p_i$ the parent of $t_i$ in $T$ for all $i\geq 1$, and define $\phi(t_i)$ conditioned on some value of $\phi(t),\phi(t_1),\ldots \phi(t_{i-1})$ to be the random variable following the uniform distribution over $\left(V_{c(t_i)} \cap N(\phi(p_i))\right) \setminus \cup_{j=1}^{i-1} \phi(t_j)$. Note that,
$$
\left|\left(V_{c(t_i)} \cap N(\phi(p_i) \right) \setminus \cup_{j=1}^{i-1} \{\phi(t_j)\}\right| \geq \deg(p_i,V_{c(t_i)}) - |c^{-1}(c(t_i))| \geq \left((1 - \gamma) - (1 - \eta)\right) |V_{c(t_i)}| \geq (\eta - \gamma) \eta n.
$$

We now discuss the spreadness of $\phi$. Fix an integer $s \leq k$ and two
sequences $t'_1,\ldots,t'_s \in V(T) \setminus \{t\}$ and $v_1,\ldots,v_s \in
V(G) \setminus \{v\}$ of distinct elements. Moreover, let us suppose that $t'_1,t'_2,\ldots,t'_s$ appear in this order in the ordering chosen above. Observe that
\begin{align*}
    \mathbb{P}\left[\bigwedge_{i=1}^{s} \phi(t'_i) = v_i\right] & = \prod_{i=1}^s \mathbb{P}\left[\phi(t'_i) = v_i \, \middle| \, \bigwedge_{j=1}^{i-1} \phi(t'_j) = v_j\right] \leq \left(\frac{1}{(\eta - \gamma)\eta n}\right)^s.
\end{align*}
The last inequality follows as for any $u \in V(G)$ and $i \in [k]$ the probability that $\phi(t_i) = u$ conditioned on any value of $\phi(t),\phi(t_1),\ldots,\phi(t_{i-1})$ is at most $\frac{1}{(\eta - \gamma)\eta n}$. In particular, the lemma follows as we have $\gamma \leq \eta/2$.\end{proof}

\section{Proof of Theorem~\ref{thm:mainthm}}\label{sec:mainsec}
\subsection{Overview}\label{sec:overview}
As briefly discussed in Section~\ref{sec:futuredirections}, our proof capitalises on several desirable properties (as collected in Lemma~\ref{lem:randompartition2}) satisfied by a random partition of the vertex-set of the host graph $G$. In this way, our proof bears resemblance to the proof in \cite{kelly2023optimal}, and we recommend the reader to investigate the proof sketch given there. 
\par To start, we split $T$ into
$O(1)$-sized edge-disjoint subtrees via \cref{treedecomp} and take a random
partition $\mathcal{R}$ of the host graph given by
Lemma~\ref{lem:randompartition2} (we will comment on the choice of parameters momentarily). Almost all of the random subsets $R$ in $\mathcal{R}$ have good enough minimum degree to contain all bounded degree trees of size $|R|$ by just applying Theorem~\ref{thm:KSS95} as a black-box. Now, we need to decide (randomly), which subtrees of $T$ will embed into which random subsets of $V(G)$. This corresponds to randomly embedding a bag-tree of the tree-splitting into $A'$, the auxiliary graph given in Lemma~\ref{lem:randompartition2} that encodes the pairs of random sets with good minimum degree. Thus, we reduce Theorem~\ref{thm:mainthm} to a weaker version of itself where the host graph $G$ is nearly complete (thanks to \ref{cond:B2}). Unfortunately, we do not have a way of directly producing the necessary random embedding even in this simpler context where the minimum degree of the host graph is extremely large. In contrast, a similar method is employed in \cite{kelly2023optimal} to embed hypergraph Hamilton cycles, but here the ``bag-tree'' of a hypergraph Hamilton cycle is simply a $2$-uniform Hamilton cycle, which is simpler to randomly embed with good spread using elementary methods.
\par To circumvent this problem, we introduce the following trick which we hope might have further applications. While applying Lemma~\ref{lem:randompartition2}, we make the sizes of the random sets an $\eps$-fraction smaller than the sizes of the subtrees they are meant to contain. This gives us extra space as we then have more random sets than subtrees we need to embed. Afterwards, using a simple random greedy strategy (see Lemma~\ref{lem:colourspread}), we can produce the necessary random embedding $\psi$ of the bag-tree into $A'$. Two problems remain: some random sets are unused by $\psi$ and the random sets that are used by $\psi$ are too small to contain the subtrees we wish to embed. We fix both of these issues by randomly reallocating all vertices of the unused random sets into the used random sets using Corollary~\ref{cor:starmatchings}. We need a fair bit of precision in this final step, which is discussed more in Section~\ref{sec:reallocation}.
\par To finish the
embedding, we need to convert $\psi$ into a random embedding $\phi : V(T)
\mapsto V(G)$. We may do this by ordering the subtrees so that each subtree
intersects the previous ones in at most one vertex, and using
Theorem~\ref{thm:rootedtree}, which is a slight strengthening of Theorem~\ref{thm:KSS95} that allows us to prescribe the location of a root vertex in advance. To illustrate, suppose $T_1,\ldots, T_{i-1}$ are already
embedded, and suppose that $\psi(T_i)$ is an empty random set large enough to
contain $T_i$. Say there exists some $t\in V(T_{i})\cap V(T_j)$ for some $j<i$,
then $\phi(t)$ is already determined, as $T_j$ is already embedded. The properties of $\psi$, coming from
\cref{lem:randompartition2} and Lemma~\ref{lem:colourspread}, guarantee that
$\phi(t)\cup \psi(T_i)$ has large minimum degree, so we may invoke
Theorem~\ref{thm:rootedtree} to extend $\phi$ to embed vertices of $T_i$ in $\psi(T_i)$, respecting the previous choice of $\phi(t)$, as desired.

\subsection{Spread distributions on star matchings}\label{sec:reallocation}
As described in Section~\ref{sec:overview}, we need a way to randomly shuffle around vertices in a random partition of $G$ to adjust the sizes of certain random sets, and we need to do this without damaging the randomness properties or the minimum degree conditions of the partition (coming from Lemma~\ref{lem:randompartition2}). The following lemma gives us a way to achieve this in the special case where each random set is meant to receive exactly one new element. It should be interpreted as applying to some auxiliary graph where one side consists of random sets, and the other side consists of vertices to be redistributed into the random sets, and an edge denotes that the vertex has good minimum degree into the random set, hence can potentially be redistributed into that random set without damaging minimum degree conditions. 
\begin{lemma}[Lemma 3.1 in \cite{kelly2023optimal}]\label{lem:graphmatchings}
There exists an absolute constant $C_{\ref*{lem:graphmatchings}}$ with the following property. If $G$ is a balanced bipartite graph on $2n$ vertices with $\delta(G)\geq 3n/4$, then there exists a random perfect matching $M$ of $G$ such that for any collection of edges $e_1,\ldots, e_s\in E(G)$, $\mathbb{P}[\bigwedge_{i\in [s]} e_i\in M]\leq (C_{\ref*{lem:graphmatchings}}/n)^s$. 
\end{lemma}
As we make each original random set multiple vertices smaller than it needs to be to be able to contain the corresponding subtree (recall Section~\ref{sec:overview}), each random set actually has to receive more than just a single new element. The next result generalises the previous lemma into this context. 
\begin{corollary}\label{cor:starmatchings} There exists an absolute constant $C_{\ref*{cor:starmatchings}}$ with the following property. Let $1/n\ll 1/k$. 
    Let $G$ be a bipartite graph with partition $(A,B)$, with $|A|=n, |B|=kn$. Suppose for each $a\in A$, $d(a, B)\geq (99/100)|B|$ and for each $b\in B$, $d(b, A)\geq (99/100)|A|$. Then, there is a random $K_{1,k}$-perfect-matching $M$ of $G$ (where the centres of the $K_{1,k}$ are embedded in $A$) such that for any collection of edges $e_1,\ldots, e_s\in E(G)$, $\mathbb{P}[\bigwedge_{i\in [s]} e_i\in M]\leq (C_{\ref*{cor:starmatchings}}/n)^s$. 
\end{corollary}
\begin{proof}
    There exists an equipartition of $B$ as $B_1,\ldots, B_k$ such that each $G[A\cup B_i]$ ($i\in[k]$) is a graph with minimum degree at least $(98/100)n$. Indeed, a random partition of $B$ would have this property with high probability (as $n/k\to \infty$, see Lemma~\ref{lem:McDAppl}). Now, Lemma~\ref{lem:graphmatchings} gives us a random perfect matching $M_i$ in each $G[A\cup B_i]$ ($i\in[k]$), and $\bigcup_{i\in [k]}M_i=:M$ is a random $K_{1,k}$-perfect-matching of $G$. $M$ clearly has the desired spread with $C_{\ref*{cor:starmatchings}}=C_{\ref*{lem:graphmatchings}}$.
\end{proof}

\subsection{Proof of Theorem~\ref{thm:mainthm}}\label{sec:mainthm}
We actually prove a stronger version of Theorem~\ref{thm:mainthm} where the location of a root vertex is specified in advance (similar to Theorem~\ref{thm:rootedtree}) as we believe this stronger result could have further applications. The unrooted version, i.e. Theorem~\ref{thm:mainthm}, follows simply by choosing $t\in V(T)$ arbitrarily, and $v\in V(G)$ uniformly at random, setting $\phi(t)=v$, and using Theorem~\ref{thm:mainthm} to complete this to a full $O(1/n)$-spread embedding of $T$.

\begin{theorem}\label{lem:approximatetreeembeddding} Let $\frac 1n \ll \frac {1}{C_*} \ll\alpha,1/\Delta$. Let $G$ be a $n$-vertex graph with
  $\delta(G)\geq (1/2+\alpha)n$. Let $T$ be a tree on $n$ vertices, with
  $\Delta(T)\leq \Delta$. Let $t \in V(T)$ and let $v \in V(G)$. The, there exists a random embedding $\phi\colon T\to G$ such that $\phi(t) = v$ with probability 1 and $\phi$ restricted to $T-t$ is $\left(\frac{\textcolor{black}{C_*}}{n}\right)$-spread.
\end{theorem}

\begin{remark}\label{rem:mainrem}
      All of the dependencies between the constants that arise from our proof are polynomial. However, we also need that $C_*$ is at least a polynomial function in $f(\alpha, \Delta)$, where $f$ is the function from Theorem~\ref{thm:rootedtree} that ultimately relies on \cite{kathapurkar2022spanning}. As \cite{kathapurkar2022spanning} does not cite an explicit bound (though their proof does not use the Szemer\'edi regularity lemma), we also formulate Theorem~\ref{lem:approximatetreeembeddding} in the above imprecise form.
\end{remark}

\begin{proof}[Proof of Theorem~\ref{lem:approximatetreeembeddding}]
  Let $C$ \aside{$C$} be a new constant such that $1/n \ll 1/C_*  \ll 1/C \ll \alpha,1/\Delta$. Let $(T_i)_{i \in [\ell]}$ be a tree-splitting of $T$ 
  obtained by \cref{treedecomp} applied with $m := C$. Notice that adding $T_*:=\{t\}$ \aside{$T_*$} to this tree-splitting produces another tree-splitting. Let $T'$ be a bag-tree of $(T_i)_{i \in [\ell]} \cup \{T_*\}$ rooted in $T_*$. Note that $T'$ has maximum degree $4C\Delta$. \aside{$T'$}

\paragraph{Step 1: Randomly partition of $V(G) \setminus \{v\}$.}

We assign to each subtree $T_i$ a colour that corresponds to its size. Formally,
define the colouring \aside{$f$} $f: V(T') \setminus \{T_*\} \to [C,4C]$ via $f(T_i) := |T_i|$ for each
subtree $T_i$. Fix an integer $K$ \aside{$K, a_c,b_c$} such that $1/C \ll 1/K \ll \alpha$. For each colour $c \in [C,4C]$, let $a_c := c -1 - K$ and $b_c := |f^{-1}(c)| +  \left\lfloor \frac{K}{32C^3} n \right\rfloor$.

We use \cref{lem:randompartition2} on $G$ with the
following parameters, $(a_i)_{i\in[k]} := (a_c)_{C \le c \le 4C}$,
$(b_i)_{i \in [k]} := (b_c)_{C \le c \le 4C}$, $\delta := \frac12$, $K := K+1$,
and $C, \alpha, v:= C,\alpha, v$.  To do so, we need only need to check
that $\sum_c a_cb_c < n$, as the other conditions follow directly from
our choice of constants. Observe
\begin{align*}
  \sum_{c=C}^{4C} a_cb_c = \sum_{i = 1 }^\ell (|T_i| - 1 - K) + \left\lfloor \frac{K}{32C^3} n \right\rfloor \sum_{c=
                  C}^{4C}(c-1-K)
                < |T| - \ell K +  \left\lfloor \frac{K}{32C^3} n \right\rfloor \frac{5C(3C+1)}{2}
                &\leq n - \frac{K}{4C}n + \frac{K}{4C}n.
\end{align*}

We can thus obtain a
partition $\mathcal{R} = (R_c^j)_{C \le c \le 4C, j \in [b_c]}$  of a subset of
$V(G) \setminus \{v\}$ and an auxiliary graph $A'$ whose vertex set $V(A')$ is a subset of
$ \mathcal{R}$, satisfying the conditions listed in \cref{lem:randompartition2}. Add to $A'$ a vertex \aside{$R_*$} $R_* = \{v\}$ adjacent to all $R \in \mathcal{R}$ such that $\delta(G[R+v]) \geq \left(\frac12 + \frac\alpha2 \right)|R+v|$. For
each colour $c \in [C, 4C]$, let us denote by $V_c := \{R_c^j \mid j \in [b_c] \} \cap V(A')$. We
define the colouring \aside{$V_c, g$} $g:V(A') \mapsto [C,4C]$ that associates the colour $c$ to all
parts in $V_c$. Formally, $\forall c \in[C,4C], \forall R \in V_c, g(R) = c$. 

\paragraph{Step 2: Construct $\psi\colon T'\to A'$.}
Conditional on a fixed outcome of $\mathcal{R}$ (and thus, $A'$), we describe \aside{$\psi$} $\psi$, which is a random embedding of $T'$ on $A'$.
We apply \cref{lem:colourspread} to $A'$ with partition $V_C\cup \ldots \cup V_{4C}$
and $T'$ coloured by $f$, with parameters $t := T_*$, $v := R_*$, $\gamma := e^{-\alpha^2 C/12}$ \aside{$\gamma,\eta$} and
$\eta := \left\lfloor \frac{K}{32C^3}\right\rfloor$ to obtain a random embedding $\psi$. To apply the lemma, we need
the following conditions to be satisfied:
  \begin{itemize}
      \item $1/n \ll \gamma = e^{-\alpha^2C/12} \ll \eta =
        \Theta\left(\frac{K}{C^3}\right)$,
      \item for all $i\in [C;4C]$, for all $u \in V(A')$, $\deg(u,V_i) \geq (1 - \gamma) |V_i|$,
      \item for all $c$, the number of subtrees coloured $c$ is at most $(1-\eta)|V_c|$.
  \end{itemize} 
  The first condition is satisfied by our constant hierarchy. Our choice of $\gamma$ and condition \ref{cond:B1} of \cref{lem:randompartition2} is tailored so that $V(A')$ satisfies the second
  condition. The third condition is less direct. Let $n_c$ be the number of subtrees
  coloured $c$ in $T'$, what we aim to show is $n_c \leq (1 - \eta)|V_c|$. By Condition~\ref{cond:B1} of \cref{lem:randompartition2} and the definition of $b_c$, we have that
  $
  |V_c| \ge (1-e^{-\frac{\alpha^2C}{12}})(n_c + \eta n)
  $. Hence,
  \begin{align*}
      (1 - \eta)|V_c| &\geq (1 - \eta)(1 - e^{-\tfrac{\alpha^2C}{12}})(n_c + \eta n) 
      \geq (1 - \eta - e^{-\tfrac{\alpha^2C}{12}})(n_c + \eta n) 
      \geq (1 - 2\eta)(1 + \frac{\eta n}{n_c})n_c \\
      &\geq (1 - 2\eta)(1 + C\eta)n_c \geq n_c
  \end{align*}
  where we used that $e^{-\tfrac{\alpha^2C}{12}} \ll \eta = \Theta\left(\frac{K}{C^3}\right)$, and for the last step that $n_c\leq \frac{n}{C}$ and $C\geq 4$.

  Recall that $\psi: V(T') \mapsto V(A')$ denotes the random embedding obtained
  from \cref{lem:colourspread}. By construction, $\psi$ restricted to $T' \setminus \{T_*\}$ is $(\frac{\textcolor{black}{{2^{15}C^6}}}{n})$-spread, and note that this spreadness condition holds independently of the values of $\mathcal{R}$ and $A'$ that we condition on. \cref{lem:colourspread} also ensures that with probability $1$, $\psi$ preserves the
  colouring given by $g$, i.e.  $\forall i \in [\ell], f(T_i) = g(\psi(T_i))$.

\paragraph{Step 3: Adjust the size of the bags.}
    In this step, we describe how to obtain a randomised partition $\mathcal{M}$ of $V(G)$. We define $\mathcal{M}$ conditional on fixed values of $\mathcal{R}$ and $\psi$. Informally, our goal is to build, for all $R_c^j \in \Ima(\psi)$, a set $M_c^j$ satisfying $R_c^j \subseteq M_c^j$ and $|M_c^j| = c-1$ while preserving the minimum degree condition given by \cref{lem:randompartition2} for the set $\Ima(\psi)$ and the edges induced by $\psi$. Formally, defining \aside{$N(R_c^j)$} $N(R_c^j) = \{ R \in \Ima(\psi) | \{\psi^{-1}(R_c^j),\psi^{-1}(R)\} \in E(T')\}$, we want $\mathcal{M}$ to satisfy the following three properties:
\begin{enumerate}[label ={{{\textbf{C\arabic{enumi}}}}}]
    \item \label{cond:C1} $\forall R_c^j \in \Ima(\psi)$, $|M_c^j| = c-1$;
    \item \label{cond:C2} $\forall R_c^j \in \Ima(\psi)$, $\delta(G[M_c^j]) \geq \left(\frac12 + \frac\alpha3 \right)|M_c^j|$;
    \item \label{cond:C3} $\forall R_c^j \in \Ima(\psi)$, $\forall R_{c'}^{j'} \in N(R_c^j)$, $\forall v \in M_c^j$, $\delta\left(G\left[M_{c'}^{j'} + v\right]\right) \geq \left(\frac12 + \frac{\alpha}{3}\right)|M_{c'}^{j'} +v|$.
\end{enumerate}

Consider the following bipartite graph \aside{$H,A,B$} $H$ with bipartition $(A,B)$ where $B = V(G) \setminus \left( \{t \} \cup \cup_{R\in\mathcal{R}} R\right)$ and $A=\mathcal{R}$. Put an edge between $v \in A$ and $R \in B$ if $\forall R' \in \{R\} \cup N(R)$, $\delta(G[R'+v]) \geq \left(\frac12 + \frac\alpha 2 \right)|R'+v|$. Note that, for all $(a,b) \in A \times B$, \ref{cond:A1} and \ref{cond:A2} imply that $d(a,B) \geq (1 -(4\Delta C +1)3e^{-\frac{\alpha^2C}{10}})|B| \geq \frac{99}{100}|B|$ and that $d(b,A) \geq (1 -(4\Delta C +1)3e^{-\frac{\alpha^2C}{10}}) \geq \frac{99}{100}|B|$.
Therefore, we may use \cref{cor:starmatchings} on $H$ with $k := K$, to associate to each $R_c^j$ a disjoint random set of $K$ elements of $B$, satisfying the spreadness property stated in the lemma (regardless of the value of $\mathcal{R}$ and $\psi$ that is being conditioned upon). Consider \aside{$L_c^j$} $L_c^j \subseteq A$, the set of random vertices that are matched by the $K_{1,K}$-perfect-matching to $R_c^j$, and define \aside{$M_c^j$} $M_c^j:= R_c^j \cup L_c^j$. Note \ref{cond:C1} is then directly satisfied. The definition of an edge in $H$
and the fact that $K/C \ll \alpha$ imply that \ref{cond:C2} and \ref{cond:C3} are also both satisfied. We define \aside{$\mathcal{M}$} $\mathcal{M}:=\bigcup_{c,j} \{M_c^j\}$. 
\par Having defined the random variable $\mathcal{M}$, we now show the following spreadness property. To clarify, $\psi$ and $\mathcal{R}$ are not considered to be fixed anymore.
\begin{claim}
    \label{cl:u_spread}
    For any $s \in \mathbb{N}$ and any function $h : \{v_1,\dots,v_s\} \mapsto \mathcal{M}$ where $\{v_1,\dots,v_s\} \subseteq V(G)$, we have $\P\left[\wedge_{i=1}^s v_i \in h(v_i)\right] \leq (\frac{\textcolor{black}{12C \cdot C_{\ref{cor:starmatchings}}}}{n})^s$.
\end{claim}

\begin{proof}
    Note that if $v \in \{v_1,\ldots,v_s\}$ then $\P[v \in h(v)] = 0$. Suppose this is not the case, and let us partition $\{v_1,\ldots,v_s\}$ into $\{x_1,\ldots,x_{s_1}\} \subseteq A$ and $\{y_1,\ldots,y_{s_2}\} \subseteq V(G) \setminus (A \cup \{v\})$. Observe that
    $$
        \P\left[\bigwedge_{i=1}^s v_i \in h(v_i)\right] = \P\left[\bigwedge_{i=1}^{s_2} y_i \in h(y_i)\right] \P\left[\bigwedge_{i=1}^{s_1} x_i \in h(x_i) \, \middle| \, \bigwedge_{i=1}^{s_2} y_i \in h(y_i)\right]
        \leq \left(\frac{12C}{n}\right)^{s_2} \left(\frac{C_{\ref{cor:starmatchings}}}{n}\right)^{s_1},$$
where we used the spreadness property from Lemma~\ref{lem:randompartition2} and Corollary~\ref{cor:starmatchings} in the last step.
        \end{proof}

\paragraph{Step 4: Embed the subtrees.} From now on, we redefine $g$ and $\psi$ as being maps to $\mathcal{M} \cup \{R_*\}$ (by composing $g$ and $\psi$ with the natural bijection $\mathcal{R} \mapsto \mathcal{M} \cup \{R_*\}$ that associates $M_c^j$ to $R_c^j$, and $R_*$ to itself). 

We fix $\phi(t):= v$, by doing so we embed $T_*$ into $R_*$.
The goal is now to embed each $T_i$ in $\psi(T_i) \in \mathcal{M}$. Note that $|\psi(T_i)| = |T_i| - 1$ for all $i \in [\ell]$ (by \ref{cond:C1}). We define $\phi$ as follows: While there exists a subtree $T_i$ that is not fully embedded into $G$, pick a subtree $T_i$
that has exactly one vertex $t_i$ already embedded say $\phi(t_i) = v_i$ and
apply~\cref{thm:rootedtree} to embed the rest of $T_i$ in $\psi(T_i)$. We can
use \cref{thm:rootedtree}, due to \ref{cond:C1}, \ref{cond:C2} and \ref{cond:C3}. This procedure is well defined because $T'$ is a tree. Let us define the \emph{native atom} of a vertex $y \in V(T)$, denoted by \aside{$T(y)$} $T(y)$, to be
the first $T_i$ that contains $y$.

\textbf{Checking spreadness.} We now prove that the random embedding $\phi$ constructed this way is $\left(\frac{\textcolor{black}{C_*}}{n}\right)$-spread. The spreadness of this embedding comes from two different randomness sources: the partition $\mathcal{M}$ via \cref{cl:u_spread}, and the random embedding $\psi$ via \cref{lem:colourspread}. 
 
Let us fix two sequences of distinct elements $y_1, \dots, y_s \in V(G - v)$ and $x_1, \dots x_s \in V(T-t)$. Let $b := |\{T(x_i) \mid i \in [s]\}|$. We may suppose, up to reordering, that $x_1,\dots,x_b$ each have distinct native atoms, this way we have $\{T_{x_1},\ldots,T_{x_b}\} = \{T_{x_1},\ldots,T_{x_s}\}$. Let us split our probability on the two sources of spreadness as follows. Set $C_0:=\textcolor{black}{12C \cdot C_{\ref{cor:starmatchings}}}$.

\begin{align*}
\P\left[\bigwedge_{i=1}^s \phi(x_i) = y_i\right] &= \sum_{h : [b] \mapsto \mathcal{M}} \P\left[ \bigwedge_{i=1}^s \phi(x_i) = y_i \middle| \bigwedge_{i=1}^s y_i \in  h(i) \right] \cdot \P\left[\bigwedge_{i=1}^s y_i \in h(i) \right]\\
&\leq \sum_{h : [b] \mapsto \mathcal{M}} \P\left[ \bigwedge_{i=1}^s \phi(x_i) = y_i \middle| \bigwedge_{i=1}^s y_i \in  h(i) \right] \cdot \left(\frac{C_0}{n}\right)^s && \text{\small by \cref{cl:u_spread}} \\
& \leq \sum_{h : [b] \mapsto \mathcal{M}} \P\left[\bigwedge_{i=1}^b \psi(T(x_i)) = h(i)\right] \cdot \left(\frac{C_0}{n}\right)^s \\
& \leq \sum_{h : [b] \mapsto \mathcal{M}} \left(\frac{2C^2}{\eta^2n}\right)^b \left(\frac{C_0}{n}\right)^s \leq |\mathcal{M}|^b \left(\frac{2C^2}{\eta^2n}\right)^b \left(\frac{C_0}{n}\right)^s \leq n^b \left(\frac{2C^2}{\eta^2n}\right)^b \left(\frac{C_0}{n}\right)^s && \text{\small by \cref{lem:randompartition2}}\\
& \leq \left(\frac{2C^2C_0}{\eta^2n}\right)^s \leq \left(\frac{\textcolor{black}{C_*}}{n}\right)^s && \text{\small $b\leq s$} 
\end{align*}

To justify the second inequality, it is sufficient to observe that $\phi(x_i) = y_i$ only if $ \psi(T(x_i)) = h(i)$. Moreover by the remark made above, $T(x_1),\dots,T(x_b)$ are all distinct, so we can indeed invoke the spreadness property of $\psi$ coming from \cref{lem:colourspread}.

\end{proof}

\bibliographystyle{amsabbrv}
\bibliography{ref}

\providecommand{\MR}[1]{}
\providecommand{\bysame}{\leavevmode\hbox to3em{\hrulefill}\thinspace}
\providecommand{\MR}{\relax\ifhmode\unskip\space\fi MR }
\providecommand{\MRhref}[2]{%
  \href{http://www.ams.org/mathscinet-getitem?mr=#1}{#2}
}
\providecommand{\href}[2]{#2}
\begin{thebibliography}{10}

\bibitem{allen2024robust}
P.~Allen, J.~B{\"o}ttcher, J.~Corsten, E.~Davies, M.~Jenssen, P.~Morris,
  B.~Roberts, and J.~Skokan, \emph{A robust {C}orr{\'a}di--{H}ajnal theorem},
  Random Structures \& Algorithms \textbf{65} (2024), 61--130.

\bibitem{anastos2023robust}
M.~Anastos and D.~Chakraborti, \emph{Robust hamiltonicity in families of
  {D}irac graphs}, arXiv preprint arXiv:2309.12607 (2023).

\bibitem{bottcher2009proof}
J.~B{\"o}ttcher, M.~Schacht, and A.~Taraz, \emph{Proof of the bandwidth
  conjecture of {B}ollob{\'a}s and {K}oml{\'o}s}, Mathematische Annalen
  \textbf{343} (2009), 175--205.

\bibitem{chernoff1952measure}
H.~Chernoff, \emph{A measure of asymptotic efficiency for tests of a hypothesis
  based on the sum of observations}, The Annals of Mathematical Statistics
  (1952), 493--507.

\bibitem{CK2009}
B.~Cuckler and J.~Kahn, \emph{Hamiltonian cycles in {D}irac graphs},
  Combinatorica \textbf{29} (2009), 299--326. \MR{2520274}

\bibitem{dirac1952}
G.~A. Dirac, \emph{Some theorems on abstract graphs}, Proc. London Math. Soc.
  (3) \textbf{2} (1952), 69--81. \MR{47308}

\bibitem{FKNP21}
K.~Frankston, J.~Kahn, B.~Narayanan, and J.~Park, \emph{Thresholds versus
  fractional expectation-thresholds}, Ann. of Math. (2) \textbf{194} (2021),
  475--495. \MR{4298747}

\bibitem{gupta2023general}
P.~Gupta, F.~Hamann, A.~M{\"u}yesser, O.~Parczyk, and A.~Sgueglia, \emph{A
  general approach to transversal versions of dirac-type theorems}, Bulletin of
  the London Mathematical Society \textbf{55} (2023), 2817--2839.

\bibitem{joos2023robust}
F.~Joos, R.~Lang, and N.~Sanhueza-Matamala, \emph{Robust hamiltonicity}, arXiv
  preprint arXiv:2312.15262 (2023).

\bibitem{joos2024counting}
F.~Joos and J.~Schrodt, \emph{Counting oriented trees in digraphs with large
  minimum semidegree}, Journal of Combinatorial Theory, Series B \textbf{168}
  (2024), 236--270.

\bibitem{KK07}
J.~Kahn and G.~Kalai, \emph{Thresholds and expectation thresholds}, Combin.
  Probab. Comput. \textbf{16} (2007), 495--502. \MR{2312440}

\bibitem{kang2024perfect}
D.~Y. Kang, T.~Kelly, D.~K{\"u}hn, D.~Osthus, and V.~Pfenninger, \emph{Perfect
  matchings in random sparsifications of dirac hypergraphs}, Combinatorica
  (2024), 1--34.

\bibitem{kathapurkar2022spanning}
A.~Kathapurkar and R.~Montgomery, \emph{Spanning trees in dense directed
  graphs}, Journal of Combinatorial Theory, Series B \textbf{156} (2022),
  223--249.

\bibitem{kelly2023optimal}
T.~Kelly, A.~M{\"u}yesser, and A.~Pokrovskiy, \emph{Optimal spread for spanning
  subgraphs of dirac hypergraphs}, Journal of Combinatorial Theory, Series B
  \textbf{169} (2024), 507--541.

\bibitem{KSS95}
J.~Koml\'{o}s, G.~N. S\'{a}rk\"{o}zy, and E.~Szemer\'{e}di, \emph{Proof of a
  packing conjecture of {B}ollob\'{a}s}, Combin. Probab. Comput. \textbf{4}
  (1995), 241--255.

\bibitem{komlos2001}
J.~Koml\'{o}s, G.~N. S\'{a}rk\"{o}zy, and E.~Szemer\'{e}di, \emph{Spanning
  trees in dense graphs}, Combin. Probab. Comput. \textbf{10} (2001), 397--416.

\bibitem{krivelevich2014robust}
M.~Krivelevich, C.~Lee, and B.~Sudakov, \emph{Robust {H}amiltonicity of {D}irac
  graphs}, Trans. Amer. Math. Soc. \textbf{366} (2014), 3095--3130.
  \MR{3180741}

\bibitem{MR02}
M.~Molloy and B.~Reed, \emph{Graph colouring and the probabilistic method},
  Algorithms and Combinatorics, vol.~23, Springer-Verlag, Berlin, 2002.

\bibitem{randomspanningtree}
R.~Montgomery, \emph{Spanning trees in random graphs}, Advances in Mathematics
  \textbf{356} (2019).

\bibitem{alp_yani_richard}
R.~Montgomery, A.~M\"{u}yesser, and Y.~Pehova, \emph{Transversal factors and
  spanning trees}, Adv. Comb. (2022).

\bibitem{park2022proof}
J.~Park and H.~Pham, \emph{A proof of the {K}ahn--{K}alai conjecture}, Journal
  of the American Mathematical Society \textbf{37} (2024), 235--243.

\bibitem{pavez2024dirac}
M.~Pavez-Sign{\'e}, N.~Sanhueza-Matamala, and M.~Stein, \emph{Dirac-type
  conditions for spanning bounded-degree hypertrees}, Journal of Combinatorial
  Theory, Series B \textbf{165} (2024), 97--141.

\bibitem{pehova2024embedding}
Y.~Pehova and K.~Petrova, \emph{Embedding loose spanning trees in 3-uniform
  hypergraphs}, Journal of Combinatorial Theory, Series B \textbf{168} (2024),
  47--67.

\bibitem{PSSS22}
H.~T. Pham, A.~Sah, M.~Sawhney, and M.~Simkin, \emph{A toolkit for robust
  thresholds}, arXiv:2210.03064 (2022).

\bibitem{SSS2003}
G.~N. S\'{a}rk\"{o}zy, S.~M. Selkow, and E.~Szemer\'{e}di, \emph{On the number
  of {H}amiltonian cycles in {D}irac graphs}, Discrete Math. \textbf{265}
  (2003), 237--250. \MR{1969376}

\bibitem{sudakov2017robustness}
B.~Sudakov, \emph{Robustness of graph properties}, Surveys in combinatorics
  2017, London Math. Soc. Lecture Note Ser., vol. 440, Cambridge Univ. Press,
  Cambridge, 2017, 372--408. \MR{3728112}

\bibitem{Ta10}
M.~Talagrand, \emph{Are many small sets explicitly small?}, Proceedings of the
  forty-second ACM symposium on Theory of computing, 2010, 13--36.

\end{thebibliography}

\appendix

\section{Random partitions}\label{app:b}

Before proving \cref{lem:randompartition2}, we state three
elementary concentration bounds that we will need. 

\begin{lemma}[Chernoff bound \cite{chernoff1952measure}]\label{chernoff} Let $X:=\sum_{i=1}^m X_i$ where $(X_i)_{i\in[m]}$ is a sequence of independent indicator random variables, and let $\Expect{X}=\mu$. For every $\gamma \in (0, 1)$, we have $\Prob{|X-\mu|\geq \gamma \mu}\leq 2e^{-\mu \gamma^2/3}$.
\end{lemma}


\begin{lemma}[\cite{gupta2023general}, Lemma 3.5]\label{lem:McDAppl} Let $\ell \in \mathbb{N}$, $0<\delta' <\delta< 1$ and $1/n, 1/\ell \ll \delta-\delta'$.
Let $G$ be a $n$-vertex graph with $\delta(G) \geq \delta n$. If $A\subseteq
V(G)$ is a vertex set of size $\ell$ chosen uniformly at random, then for every
$v \in V(G)$ we have 

\begin{equation*}
\Prob{\deg(v,A) < \delta' \ell}\leq 2 \exp (-\ell(\delta - \delta')^2/2).
\end{equation*}
\end{lemma}

Finally, we need the following result due to McDiarmid, which appears in the textbook of Molloy and Reed~\cite[Chapter 16.2]{MR02}.
Here, a \textit{choice} is the position that a particular element gets mapped to in a permutation.
\begin{lemma}[McDiarmid's inequality for random permutations]\label{lem:mcdiarmidperm}Let $X$ be a non-negative random variable determined by a uniformly sampled random permutation $\pi$ of $[n]$ such that the following holds for some $c,r>0$:
\begin{enumerate}
    \item Interchanging two elements of $\pi$ can affect the value of $X$ by at most $c$
    \item For any $s$, if $X\geq s$ then there is a set of at most $rs$ choices whose outcomes certify that $X\geq s$.
\end{enumerate}
Then, for any $0\leq t\leq \mathbb{E}[X]$,
 
$$\mathbb{P}(|X-\mathbb{E}[X]|\geq t + 60c\sqrt{r\mathbb{E}[X]})\leq 4\exp(-t^2/(8c^2r\mathbb{E}[X])).$$
    
\end{lemma}

\subsection{Proof of \cref{lem:randompartition2}}

\begin{proof}[Proof of \cref{lem:randompartition2}]
    Let us consider a uniformly sampled partition $\mathcal{R}$ of  a subset of $V(G)
    \setminus v$ with exactly $b_c$ parts of size $a_c$ for every $C \le c \le
    4C$. We will show that $\mathcal{R}$ conditional on the \ref{cond:A0},   \ref{cond:A1}, \ref{cond:A2}, \ref{cond:B1}, \ref{cond:B2} and \ref{cond:B3} has the desired spreadness property. First, we show the intersection of the events \ref{cond:A0},   \ref{cond:A1}, \ref{cond:A2}, \ref{cond:B1}, \ref{cond:B2} and \ref{cond:B3} has probability $1-o(1)$. 
    \par We denote by $\mathcal{R}_c := \bigcup_{j=1}^{b_c}
    R_c^j$ the set of all vertices contained in a set of size $a_c$. For $u\in V(G)$ and $R_c^j\in \mathcal{R}$, we say that $(u,R_c^j)$ is a \textit{good pair} if $\delta(G[R_c^j+u])\geq (\delta+2\alpha/3)|R_c^j+u|$, and we say it is a \textit{bad pair} if it is not a good pair. For any $u$ and $R_c^j$, we have that $\mathbb{P}[\text{$(u,R_c^j)$ is a good pair}]\geq 1-e^{-\alpha^2|R_c^j|/10} \geq  1-e^{-\alpha^2C/10}$
    by Lemma~\ref{lem:McDAppl}. 
    \par Consider a fixed $R_c^j \in \mathcal{R}$ and a fixed $d \in [C,4C]$, and let $X_{R_c^j,d}$ be the random variable counting the number of good pairs $(u,R_c^j)$ where $u \in \mathcal{R}_d \setminus R_c^j$. We can naturally view the sets in $\mathcal{R} \setminus \{R_c^j\}$ as being obtained by looking at consecutive intervals of given lengths in a random permutation $\pi$ of $V(G) \setminus R_c^j$. Then, interchanging two elements from $\pi$ affects $X_{R_c^j,d}$ by at most $2$, and for each $s$ if $X_{R_c^j,d}\geq s$, we can certify this with specifying $4Cs$ coordinates of $\pi$ (as each $a_d\geq 4C$). Thus Lemma~\ref{lem:mcdiarmidperm} gives us the following inequality (using that $\E[X_{R_c^j,d}] \geq \left(1 - e^{-\alpha^2C/10}\right)a_db_d$), $$
        \P\left[X_{R_c^j,d} \leq \left(1 - 3e^{-\alpha^2C/10}\right)a_db_d\right] \leq 4e^{-\tfrac{e^{-\alpha^2C/10}}{100}\eta n} = o(n^{-1}).$$
     By taking a union bound on all sets $R_c^j$ in the
     partition $\mathcal{R}$, we obtain, with probability $1 - o(1)$, that for every
     set $R_c^j \in \mathcal{R}$ and for every $d \in [C,4C]$, the number of bad
      pairs $(u,R_c^j)$, with $u \in \mathcal{R}_d$, is at most $3e^{-\alpha^2C/10}b_d$. Summing over all $d$, we see that \ref{cond:A1} is satisfied with probability $1 - o(1)$. A similar reasoning over the random variable that counts the number of good pairs $(u,R_c^j)$ for a fixed $u$ instead of a fixed $R_c^j$, implies that  \ref{cond:A2} is satisfied with probability $1 - o(1)$.
     
     
     For a given set $R_d^k \in \mathcal{R}$ where $R_d^k \subseteq \mathcal{R}_d$, a union bound shows that
         $\P[\forall u \in R_d^k, (R_c^j,u) \text{ is a good pair}] \geq 1 - e^{-\alpha^2C/20}$ ($*$). Define $A'$ to consist of the good sets $R^j_c$. That \ref{cond:B1} and \ref{cond:B2} hold with high probability follows from $(*)$ to compute the corresponding expectation and invoking \ref{lem:mcdiarmidperm} as before.

     Let $E$ denote the conjunctions of all of \ref{cond:A0},   \ref{cond:A1}, \ref{cond:A2}, \ref{cond:B1}, \ref{cond:B2} and \ref{cond:B3}, noting $E$ has probability $1-o(1)$. Consider $s$ distinct vertices $u_1,\ldots,u_s \in V(G)$ and a function $f:\{u_1\ldots,u_s\}\mapsto \mathcal{R}_c^j$,
    \begin{align*}
        \P\left[\bigwedge_{i=1}^s u_i \in f(u_i) \middle| E \right] & \leq \P(E)^{-1} \cdot \P\left[\bigwedge_{i=1}^s u_i \in f(u_i) \right] 
        \leq (1 - o(1)) \left(\frac{4eC}{n}\right)^s \leq \left(\frac{12C}{n}\right)^s,
    \end{align*}
    where the second inequality follows from \cref{lem:perm_spread} for $x_i = u_i$ and $L_i = f(u_i)$. \end{proof}

\end{document}